\pdfoutput=1
\documentclass[letterpaper,11pt]{amsart}
\usepackage{ifdraft}

\PassOptionsToPackage{final,hidelinks}{hyperref}
\PassOptionsToPackage{final}{graphicx}
\usepackage{hyperref}

\usepackage{nima}

\usepackage{tikz}

\tikzset{
  vertex/.style={circle,minimum size=0.15cm,inner sep=0,fill=black},
  minivertex/.style={circle,minimum size=0.05cm,inner sep=0,fill=black},
  ensquare/.style={rectangle,minimum size=0.3cm,inner sep=0,draw},
  encircle/.style={circle,minimum size=0.25cm,inner sep=0,draw},
  thickeraser/.style={line width=2.4pt, white}
}

\usepackage{amsmath,amsfonts,amssymb,amsthm}
\usepackage{url}
\usepackage{graphicx,color}
\usepackage{enumerate}

\DeclareMathOperator{\CAT}{CAT}

\newenvironment{subproof}[1][\proofname]{%
  \begin{proof}[#1]%
}{%
  \end{proof}%
}

\title{$7$-location, weak systolicity and isoperimetry}

\author{Nima Hoda}
\address{Department of Mathematics, Cornell University \\
  Ithaca, NY 14853, USA}
\email{nima@nimahoda.net}

\author{Ioana-Claudia Laz{\u a}r}
\address{
Dept. of Mathematics\\
Politehnica University of Timi\c{s}oara\\
Victoriei Square $2$, $300006$-Timi\c{s}oara, Romania}
\email{ioana.lazar@upt.ro}

\date{\today}



\begin{document}

\begin{abstract}
  $m$-location is a local combinatorial condition for flag simplicial
  complexes introduced by Osajda.  Osajda showed that simply connected
  $8$-located locally $5$-large complexes are hyperbolic.  We treat
  the nonpositive curvature case of $7$-located locally $5$-large
  complexes.

  We show that any minimal area disc diagram in a $7$-located locally
  $5$-large complex is itself $7$-located and locally $5$-large.  We
  define a natural $\CAT(0)$ metric for $7$-located disc diagrams and
  use this to prove that simply connected $7$-located locally
  $5$-large complexes have quadratic isoperimetric function.  Along
  the way, we prove that locally weakly systolic complexes are
  $7$-located locally $5$-large.
\end{abstract}

\maketitle

\tableofcontents

\section{Introduction}

Metric and combinatorial notions of negative and nonpositive curvature
have seen significant study and application in recent decades,
especially in the areas of geometric group theory and metric graph
theory \cite{Gromov:1987, Chepoi:2000, Haglund:2003,
  Januszkiewicz:2006, Bandelt_Pesch:Dismantling_abs_retracts:1989,
  Chapolin_Chepoi_Genevois_Hirai_Osajda:Helly:2025, Bandelt:1988,
  Hoda:quadric_complexes:2020}.  These notions are often given as
local conditions that have global implications on spaces.  Of
particular note on the combinatorial side are the developments of
$\CAT(0)$-cubical complexes (also known as median graphs)
\cite{Gromov:1987, Chepoi:2000}, systolic complexes (also known as
bridged graphs) \cite{Chepoi:2000, Haglund:2003, Januszkiewicz:2006}
and Helly graphs
\cite{Chapolin_Chepoi_Genevois_Hirai_Osajda:Helly:2025}.

The combinatorial nonpositive curvature condition of
\emph{$m$-location} for simplicial complexes was introduced by Osajda
in \cite{Osajda:2015} as a method for proving hyperbolicity for
$3$-manifolds.  While local systolicity of a complex depends only on
the combinatorics around a single vertex, $m$-location depends on the
combinatorics around pairs of adjacent vertices, allowing more
flexibility.  Osajda proved that simply connected $8$-located locally
$5$-large complexes are Gromov hyperbolic and used this to obtain a
new solution to a problem of Thurston.  The second named author
studied a version of $8$-location, suggested by Osajda
\cite[Subsection 5.1]{Osajda:2015} and showed that simply connected,
$8$-located simplicial complexes are Gromov hyperbolic
\cite{Lazar:2015}.  In later work, the second named author introduced
another combinatorial curvature condition, called the $5/9$-condition,
and showed that the complexes which satisfy it, are $8$-located and
therefore Gromov hyperbolic as well \cite{Lazar:2020}.

The current paper continues the study of $m$-located complexes. As
noted by Osajda \cite{Osajda:2015}, systolic complexes are
$7$-located.  We strengthen this statement by showing that weakly
systolic complexes are $7$-located.  This provides many more examples
of $7$-located locally $5$-large complexes, including thickenings of
$\CAT(0)$ cube complexes whose links do not contain induced $4$-cycles
\cite{Osajda:2013-2}.

\begin{mainthm}[\Thmref{weak_syst_loc}]\mainthmlabel{loc_weak_syst}
  Locally weakly systolic complexes are $7$-located locally $5$-large.
\end{mainthm}

We also prove the following two theorems which show that $7$-location
and local $5$-largeness have significant global consequences: disc
diagrams with particularly nice combinatorial and geometric
properties.

\begin{mainthm}[Theorem~\ref{5.7}]\mainthmlabel{minimal_discs}
  Minimal area disc diagrams in $7$-located locally $5$-large
  complexes are $7$-located locally $5$-large.
\end{mainthm}

Let $P_n$ be the regular Euclidean $n$-gon subdivided into $n$
congruent triangles meeting at the barycenter of $P_n$.  Let $T_n$ be
a triangle of $P_n$.

\begin{mainthm}[Theorem~\ref{5.6}]\mainthmlabel{cat0_discs}
  Let $D$ be a $7$-located simplicial disc.  There is a natural way to
  metrize the triangles of $D$ using only isometry types of $T_4$,
  $T_5$ and $T_6$ such that the resulting piecewise Euclidean triangle
  complex is $\CAT(0)$.  The choice of metrization for each triangle
  depends only on the local combinatorics of $D$ around the triangle.
\end{mainthm}

From \Mainthmref{minimal_discs} and \Mainthmref{cat0_discs} we obtain
the following global coarse geometric and algebraic consequences.

\begin{maincor}[\Corref{isoperfunc}]\maincorlabel{isoper}
  A simply connected $7$-located locally $5$-large complex has
  quadratic isoperimetric function.  Consequently, the fundamental
  group of a $7$-located locally $5$-large complex has quadratic
  isoperimetric function.
\end{maincor}

\subsection{Structure of the paper}

In \Secref{prelims} we present basic definitions and notation, in
\Secref{weak_systolic} we prove \Mainthmref{loc_weak_syst}, in
\Secref{cat0_discs} we prove \Mainthmref{cat0_discs} and in
\Secref{minimal_discs} we prove \Mainthmref{minimal_discs}.

\section{Preliminaries}
\seclabel{prelims}

Let $X$ be a simplicial complex.  If $v_1, v_2, \ldots, v_k$ are
vertices of $X$ then we use the notation
$\langle v_1, v_2, \ldots, v_k \rangle$ to denote the simplex spanned
on the vertices $v_1, v_2, \ldots, v_k$, should it exist.  A
subcomplex $L$ of $X$ is \emph{full} if any simplex of $X$ whose
vertices are contained in $L$ is contained in $L$.  For distinct
vertices $v$ and $v'$ of $X$ we write $v \sim v'$ if $v$ and $v'$ are
adjacent in the $1$-skeleton of $X$ and otherwise we write
$v \nsim v'$.  The simplicial complex $X$ is \emph{flag} if any finite
set of pairwise adjacent vertices of $X$ spans a simplex of $X$.
Going forward we will always assume that $X$ is a flag complex.

Unless otherwise stated, when referring to distances between vertices
of a simplicial complex we will mean the \emph{graph metric}, i.e.,
the number of edges in the shortest $1$-skeleton path joining the
vertices.

A \emph{cycle} or \emph{loop} $\gamma$ in $X$ is a subcomplex of $X$
isomorphic to a subdivision of $S^{1}$.  We may denote a cycle
$\gamma$ by a sequence $(v_1,v_2, \ldots,v_k)$ where
$\{v_1,v_2, \ldots,v_k\}$ is the set of vertices of $\gamma$ and where
$\langle v_i, v_{i+1}\rangle$ is an edge of $\gamma$ for each $i$,
indices taken modulo $k$.  The \emph{length} of $\gamma$, denoted by
$|\gamma|$, is the number of edges of $\gamma$.  A \emph{$k$-wheel} in
$X$ is a subcomplex $W = (v_0; v_1, v_2, \ldots, v_k)$ of $X$ where
$(v_1, \ldots, v_k)$ is a cycle and
$\langle v_0, v_i, v_{i+1} \rangle$ is a triangle of $W$ for each $i$,
indices taken modulo $k$.  That is, the $k$-wheel $W$ is a simplicial
cone on its \emph{boundary} cycle
$\partial W = (v_1, v_2, \ldots, v_k)$.

Let $\gamma$ be a cycle in $X$.  A \emph{disc diagram} $(D,f)$ for
$\gamma$ is a nondegenerate simplicial map $f : D \rightarrow X$ where
$D$ is a triangulated disc, and $f | _{\partial D}$ maps $\partial D$
isomorphically onto $\gamma$.  The disc diagram $(D,f)$ has
\emph{minimal area} if $D$ has the least possible number of triangles
among all disc diagrams for $\gamma$.

By van Kampen's Lemma, every nullhomotopic cycle in $X$ has a disc
diagram.  It is not difficult to see that when $X$ is a flag
simplicial complex, any minimal area disc diagram in $X$ is also a
flag simplicial complex.

\begin{figure}[h]
  \begin{center}
    \includegraphics[height=4cm]{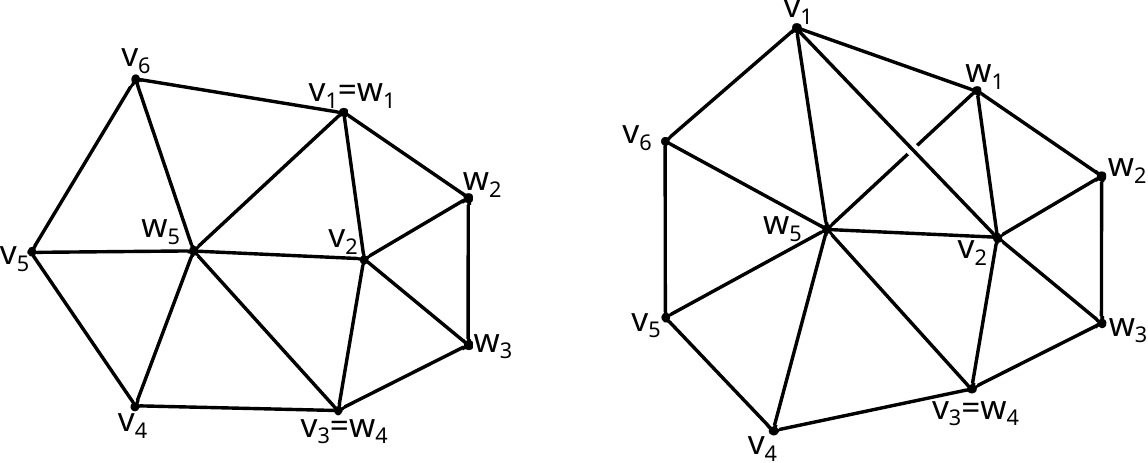}
  \end{center}
  \caption{A planar $(6,5)$-dwheel and a nonplanar
    $(6,5)$-dwheel.}
  \figlabel{dwheels}
\end{figure}

A $(k,l)$-dwheel $W = W_{1} \cup W_{2}$ in $X$ is the union of two
wheels $W_1 = (w_{l}; v_{1},v_{2}, v_{3},\ldots,$ $ v_{k})$
and $W_2 = (v_{2}; w_{1}, w_{2}, \ldots, w_{l})$ with
$v_{3} = w_{l-1}$ and either $v_{1}=w_{1}$ or $v_{1} \sim w_{1}$.  See
\Figref{dwheels}.  If $v_{1}=w_{1}$, we call $W$ a \emph{planar
  dwheel}. If $v_{1} \sim w_{1}$, we call $W$ a \emph{nonplanar
  dwheel}.  The \emph{boundary} $\partial W$ of $W$ is the cycle
$(w_1,w_2,\ldots,w_{l-1}=v_3,v_4,\ldots,v_k)$.  The boundary
$\partial W$ is the subcomplex of $W$ induced on its vertices
excluding the wheel centers $w_l$ and $v_2$.  Note that $\partial W$
has length $k+l-4$ if $W$ is planar and $k+l-3$ if it is nonplanar.

The \emph{link} of $X$ at a vertex $v$, denoted $X_v$, is the
subcomplex of $X$ consisting of all simplices of $X$ that are disjoint
from $v$ and that, together with $v$, span a simplex of $X$.

\begin{defn}
  A simplicial complex $X$ is $m$-\emph{located} for $m \geq 4$ if it
  is flag and whenever a dwheel subcomplex $W = W_1 \cup W_2$ of $X$
  satisfies the conditions
  \begin{enumerate}
  \item $\partial W$ has length at most $m$, and
  \item the wheels $W_1$ and $W_2$ are full subcomplexes of $X$,
  \end{enumerate}
  the dwheel subcomplex $W$ is contained in the link $X_v$ of some
  vertex $v$.
\end{defn}

\begin{rmk}\rmklabel{link_or_ball}
  In the definition of $m$-locatedness, we may weaken the requirement
  that $W$ is contained in the link $X_v$ of a vertex by requiring
  only that $W$ is contained in the $1$-ball $B_1(v)$ centered at a
  vertex.  This results in an equivalent definition since if $W$ is
  contained in $B_1(v)$ then either it is contained in $X_v$ or
  $v \in W$, which implies that the wheels of $W$ are not both full.
\end{rmk}

\begin{defn}
  Let $\sigma$ be a simplex of $X$.  A flag simplicial complex is
  $k$-\emph{large} if there are no full $j$-cycles in $X$, when
  $4 \le j \le k-1$.  We say $X$ is \emph{locally} $k$-\emph{large} if
  all its links are $k$-large.  Note that if $X$ is $k$-large then it
  is locally $k$-large.
\end{defn}

\begin{figure}[h]
  \centering
  \begin{tikzpicture}

    \coordinate (c) at (0,0);
    \coordinate (v1) at (360-36:1.5);
    \coordinate (v2) at (72-36:1.5);
    \coordinate (v3) at (144-36:1.5);
    \coordinate (v4) at (216-36:1.5);
    \coordinate (v5) at (288-36:1.5);
    \coordinate (a) at (36-36:1.618*1.5);

    \begin{scope}[thick]
      \draw (v1) -- (v2) -- (v3) -- (v4) -- (v5) -- (v1);
      \draw (c) -- (v1);
      \draw (c) -- (v2);
      \draw (c) -- (v3);
      \draw (c) -- (v4);
      \draw (c) -- (v5);
      \draw (v1) -- (a) -- (v2);
    \end{scope}
    
    \node[vertex,label={right:$a$}] at (a) {};
    \node[vertex,label={right:$c$}] at (c) {};
    \node[vertex,label={below:$v_1$}] at (v1) {};
    \node[vertex,label={above:$v_2$}] at (v2) {};
    \node[vertex,label={left:$v_3$}] at (v3) {};
    \node[vertex,label={left:$v_4$}] at (v4) {};
    \node[vertex,label={left:$v_5$}] at (v5) {};
  \end{tikzpicture}
  \caption{An extended $5$-wheel.}
  \figlabel{extended_five_wheel}
\end{figure}

For the purposes of this paper, we need only refer to local weak
systolicity, a condition which together with simple connectedness is
equivalent to weak systolicity by the local-to-global theorem of
Chepoi and Osajda \cite[Theorem~A]{Chepoi:2015}.  In order to define
local weak systolicity, we first need to introduce some terminology.
An \emph{extended $5$-wheel}
$\widehat{W}_5 = (c; v_1, v_2, v_3, v_4, v_5; a)$ in a flag simplicial
complex $X$ is a subcomplex consisting of a $5$-wheel
$W_5 = (c; v_1, v_2, v_3, v_4, v_5)$ together with an additional
vertex $a$ that spans a triangle $\langle a, v_1, v_2 \rangle$ with
$v_1$ and $v_2$ but that is not adjacent to any other vertex of $W_5$.
See \Figref{extended_five_wheel}.

\begin{defn}
  A flag simplicial complex $X$ is \emph{locally weakly systolic} if
  it is $5$-large and every extended $5$-wheel is contained in the
  link of a vertex.  This latter condition is called the
  \emph{$\widehat{W}_5$-condition}.
\end{defn}

\section{Weakly systolic complexes are $7$-located} 
\seclabel{weak_systolic}

In this section we prove that locally weakly systolic complexes are
$7$-located. Because locally weakly systolic complexes are $5$-large,
they thus provide a large class of examples of $7$-located locally
$5$-large complexes.

\begin{figure}[h]
  \centering
  \begin{tikzpicture}
    \coordinate (w5) at (-1,0);
    \coordinate (v2) at (1,0);
    \coordinate (v3) at (0,-4/3);
    \coordinate (v1) at (0,4/3);
    \coordinate (v4) at (-2,-3/4);
    \coordinate (v5) at (-2,3/4);
    \coordinate (w3) at (2,-3/4);
    \coordinate (w2) at (2,3/4);

    \begin{scope}[thick]
      \draw (v1) -- (v2) -- (v3) -- (v4) -- (v5) -- (v1);
      \draw (w5) -- (v1);
      \draw (w5) -- (v2);
      \draw (w5) -- (v3);
      \draw (w5) -- (v4);
      \draw (w5) -- (v5);
      \draw (v1) -- (w2) -- (w3) -- (v3);
      \draw (v2) -- (w2);
      \draw (v2) -- (w3);
    \end{scope}
    
    \node[vertex,label={left:$w_5$}] at (w5) {};
    \node[vertex,label={above:$v_1=w_1$}] at (v1) {};
    \node[vertex,label={right:$v_2$}] at (v2) {};
    \node[vertex,label={below:$v_3=w_4$}] at (v3) {};
    \node[vertex,label={left:$v_4$}] at (v4) {};
    \node[vertex,label={left:$v_5$}] at (v5) {};
    \node[vertex,label={right:$w_2$}] at (w2) {};
    \node[vertex,label={right:$w_3$}] at (w3) {};
  \end{tikzpicture}
  \caption{A planar $(5,5)$-dwheel.}
  \figlabel{p_five_five_dwheel}
\end{figure}
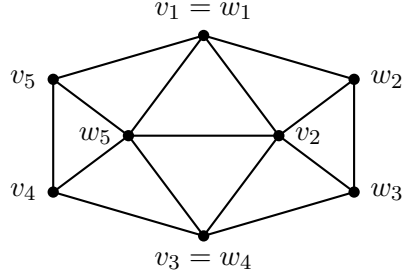

\begin{thm}\thmlabel{weak_syst_loc}
  Let $X$ be a locally weakly systolic complex. Then $X$ is
  $7$-located.
\end{thm}

\begin{proof}
  Let $W = W_1 \cup W_2$ be a $(k,l)$-dwheel subcomplex of $X$ of
  boundary length at most $7$ such that $W_1$ and $W_2$ are full
  subcomplexes.  By $5$-largeness of $X$, we have $k,l \ge 5$.  Then,
  since $W$ has boundary length at most $7$, either $W$ is a planar or
  nonplanar $(5,5)$-dwheel or $W$ is a planar $(5,6)$-dwheel.  By
  \Rmkref{link_or_ball} it suffices to show that $W$ is contained in
  the $1$-ball centered at a vertex of $X$.

  \begin{claim}
    If $W = W_1 \cup W_2$ is a planar $(5,5)$-dwheel then $W$ is
    contained in a $1$-ball of $X$.
  \end{claim}
  \begin{subproof}
    Let
    $W_{1} \cup W_{2} = (w_{5};v_{1},v_{2},v_{3},v_{4},v_{5}) \cup
    (v_{2};w_{1}=v_{1},w_{2},w_{3},w_{4}=v_{3},w_{5})$. Then $W_{1}$
    and $W_{2}$ are full.  See \Figref{p_five_five_dwheel}.

    Consider $W_1 \cup \{w_2\}$.  Note that $w_2 \sim v_1$ and
    $w_2 \sim v_2$.  Also, since $W_2$ is a full $k$-wheel, the
    vertices $w_2$ and $w_5$ are distinct and $w_2 \nsim w_5$ and
    $w_2 \nsim v_3$.  If we had $w_2 \sim v_4$ then
    $(w_2, v_2, w_5, v_4)$ would be a $4$-cycle and then $5$-largeness
    of $X$ would contradict fullness of either $W_1$ or $W_2$.  So
    $w_2 \nsim v_4$.  We also have $w_2 \nsim v_5$ arguing otherwise
    similarly using the $4$-cycle $(w_2,v_2,w_5,v_5)$.  Thus
    $W_1 \cup \{w_2\} = (w_5; v_1, v_2, v_3, v_4, v_5; w_2)$ is an
    extended $5$-wheel.  Thus, by the $\widehat{W_{5}}$-condition,
    there is a vertex $\bar w_2$ that is adjacent to but distinct from
    every vertex of $W_1 \cup \{w_2\}$.  If $\bar w_2 = w_3$ then
    $W_1 \cup W_2 = W_1 \cup \{w_2,w_3\}$ is contained in the $1$-ball
    centered at $\bar w_2$ and we are done.  So we may assume that
    $\bar w_2 \notin W_1 \cup W_2$.  By a similar argument, there is a
    vertex $\bar w_3 \notin W_1 \cup W_2$ that is adjacent to every
    vertex of $W_1 \cup \{w_3\}$.  If $\bar w_2 = \bar w_3$ then
    $W_1 \cup W_2 = W_1 \cup \{w_2, w_3\}$ is contained in the
    $1$-ball centered at $\bar w_2$ and we are done.  So we may assume
    that $\bar w_2 \neq \bar w_3$.  Then
    $(v_1, \bar w_2, v_3, \bar w_3)$ is a $4$-cycle and $5$-largeness
    of $X$ and fullness of $W_1$ imply that $\bar w_2 \sim \bar w_3$.
    Finally, the $5$-largeness of $X$ in consideration of the
    $4$-cycle $(\bar w_2, w_2, w_3, \bar w_3)$ implies that
    $W_1 \cup W_2$ is contained in the $1$-ball centered at either
    $\bar w_2$ or $\bar w_3$.
  \end{subproof}

  \begin{claim}\claimlabel{weaksys_ff_dwheel}
    If $W = W_1 \cup W_2$ is a nonplanar $(5,5)$-dwheel then $W$ is
    contained in a $1$-ball of $X$.
  \end{claim}
  \begin{subproof}
    Let $W_{1} \cup W_{2} = (w_{5};v_{1},v_{2},v_{3},v_{4},v_{5})$
    $\cup (v_{2};w_{1},w_{2},w_{3},w_{4}=v_{3},w_{5})$.  See
    \Figref{np_5_5_dwheel}.

    \begin{figure}[h]
      \centering
      \begin{tikzpicture}
        \coordinate (w5) at (-1,0);
        \coordinate (v2) at (1,0);
        \coordinate (v3) at (0,-4/3);
        \coordinate (v1) at (-2/3,4/3);
        \coordinate (w1) at (2/3,4/3);
        \coordinate (v4) at (-2,-3/4);
        \coordinate (v5) at (-2,3/4);
        \coordinate (w3) at (2,-3/4);
        \coordinate (w2) at (2,3/4);

        \begin{scope}[thick]
          \draw (w5) -- (w1);
          \draw (w1) -- (w2) -- (w3) -- (v3);
          \draw (v2) -- (w1);
          \draw (v2) -- (w2);
          \draw (v2) -- (w3);
          \draw (v1) -- (w1);
          \draw[thickeraser] (v1) -- (v2); \draw (v1) -- (v2);
          \draw (v2) -- (v3) -- (v4) -- (v5) -- (v1);
          \draw (w5) -- (v1);
          \draw (w5) -- (v2);
          \draw (w5) -- (v3);
          \draw (w5) -- (v4);
          \draw (w5) -- (v5);
        \end{scope}

        \node[vertex,label={left:$w_5$}] at (w5) {};
        \node[vertex,label={above:$v_1$}] at (v1) {};
        \node[vertex,label={right:$v_2$}] at (v2) {};
        \node[vertex,label={below:$v_3=w_4$}] at (v3) {};
        \node[vertex,label={left:$v_4$}] at (v4) {};
        \node[vertex,label={left:$v_5$}] at (v5) {};
        \node[vertex,label={above:$w_1$}] at (w1) {};
        \node[vertex,label={right:$w_2$}] at (w2) {};
        \node[vertex,label={right:$w_3$}] at (w3) {};

        \node[encircle] at (w5) {};
        \node[encircle] at (v1) {};
        \node[encircle] at (v2) {};
        \node[encircle] at (v3) {};
        \node[encircle] at (v4) {};
        \node[encircle] at (v5) {};
        \node[encircle] at (w3) {};

        \node[ensquare] at (v2) {};
        \node[ensquare] at (w1) {};
        \node[ensquare] at (w2) {};
        \node[ensquare] at (w3) {};
        \node[ensquare] at (v3) {};
        \node[ensquare] at (w5) {};
        \node[ensquare] at (v4) {};
      \end{tikzpicture}
      \caption{A nonplanar $(5,5)$-dwheel from the proof of
        \Claimref{weaksys_ff_dwheel}.  The encircled vertices form an
        extended $5$-wheel $\widehat{W}_1$ contained in the link $X_p$
        of the vertex $p$.  The ensquared vertices form an extended
        $5$-wheel $\widehat{W}_2$ contained in $X_q$.}
      \figlabel{np_5_5_dwheel}
    \end{figure}
    
    Let
    $\widehat{W}_{1} = W_{1} \cup \{w_{3}\} =
    (w_{5};v_{1},v_{2},v_{3},v_{4},v_{5};w_{3})$. We will show that
    $\widehat{W}_{1}$ is an extended $5$-wheel. Note that $W_{1}$ and
    $W_2$ are full because they form a dwheel.  This implies that
    $w_3$ is distinct from the vertices of $W_2$.  It remains to show
    that $w_{3}$ does not form an edge with any of $w_5$, $v_4$, $v_5$
    or $v_1$.  That $w_3 \nsim w_5$ follows by fullness of $W_2$.  The
    remaining cases would each introduce a $4$-cycle:
    \begin{itemize}
    \item $w_{3} \sim v_{4}$ introduces the
      $4$-cycle $(w_{3},v_{2},w_{5},v_{4})$.

    \item $w_{3} \sim v_{5}$ introduces the $4$-cycle
      $(w_{3},v_{3},w_{5},v_{5})$.

    \item $w_{3} \sim v_{1}$ introduces the $4$-cycle
      $(w_{3},v_{3},w_{5},v_{1})$.
    \end{itemize}
    The $5$-largeness of $X$ applied to any of these $4$-cycles would
    contradict fullness of $W_1$ or $W_2$.  Thus $\widehat{W}_{1}$ is
    an extended $5$-wheel and so, by the $\widehat{W}_{5}$-condition,
    there exists a vertex $p$ of $X$ such that
    $\widehat{W}_{1} \subset X_p$.  Fullness of $W_2$ implies that $p$
    is distinct $w_1$ and $w_2$ so $p$ is distinct from the vertices
    of $W = W_1 \cup W_2$.  By a similar argument, we have that
    $\widehat{W}_{2} = W_{2} \cup \{v_4\} =
    (v_{2};w_{1},w_{2},w_{3},w_{4},w_{5};v_{4})$ is an extended
    $5$-wheel contained in $X_q$ for some vertex $q$ of
    $X \setminus W$.  If $p = q$ then $W \subset X_p$ and we are done
    so we may assume that $p$ and $q$ are distinct.  Thus we have a
    $4$-cycle $(p,q,w_1,v_1)$ to which we can apply
    $5$-largeness.  Without loss of generality, we have $p \sim w_1$
    so that we have a $4$-cycle $(p,w_1,w_2,w_3)$.  But
    $w_1 \nsim w_3$ by fullness of $W_2$ so, applying
    $5$-largeness, we have $p \sim w_2$ and so
    $W = W_1 \cup W_2 \subset B_1(p)$.
  \end{subproof}

  \begin{claim}\claimlabel{weaksys_fs_dwheel}
    If $W = W_1 \cup W_2$ is a planar $(5,6)$-dwheel then $W$ is
    contained in a $1$-ball of $X$.
  \end{claim}
  \begin{subproof}
    Let
    $W_{1} \cup W_{2} = (w_{6};v_{1},v_{2},v_{3},v_{4},v_{5}) \cup
    (v_{2};w_{1}=v_{1},w_{2},w_{3},$ $w_{4},w_{5}=v_{3},w_{6})$. Then
    $W_{1}$ and $W_{2}$ are full. Let
    $\widehat{W}_{1}=W_{1} \cup \{w_{2}\} =
    (w_{6};v_{1},v_{2},v_{3},v_{4},v_{5};w_{2})$. Let
    $\widehat{W}_{1}' =W_{1} \cup \{w_{4}\} =
    (w_{6};v_{1},v_{2},v_{3},v_{4},v_{5};w_{4})$.  See
    \Figref{p_five_six_dwheel}.

    \begin{figure}[h]
      \centering
      \begin{tikzpicture}
        \coordinate (w6) at (-1,0);
        \coordinate (v2) at (1,0);
        \coordinate (v3) at (0,-4/3);
        \coordinate (v1) at (0,4/3);
        \coordinate (v4) at (-2,-3/4);
        \coordinate (v5) at (-2,3/4);
        \coordinate (w3) at (3,0);
        \coordinate (w2) at (2,4/3);
        \coordinate (w4) at (2,-4/3);

        \begin{scope}[thick]
          \draw (v1) -- (v2) -- (v3) -- (v4) -- (v5) -- (v1);
          \draw (w6) -- (v1);
          \draw (w6) -- (v2);
          \draw (w6) -- (v3);
          \draw (w6) -- (v4);
          \draw (w6) -- (v5);
          \draw (v1) -- (w2) -- (w3) -- (w4) -- (v3);
          \draw (v2) -- (w2);
          \draw (v2) -- (w3);
          \draw (v2) -- (w4);
        \end{scope}
        
        \node[vertex,label={left:$w_6$}] at (w6) {};
        \node[vertex,label={above:$v_1=w_1$}] at (v1) {};
        \node[vertex,label={below:$v_2$}] at (v2) {};
        \node[vertex,label={below:$v_3=w_5$}] at (v3) {};
        \node[vertex,label={left:$v_4$}] at (v4) {};
        \node[vertex,label={left:$v_5$}] at (v5) {};
        \node[vertex,label={right:$w_2$}] at (w2) {};
        \node[vertex,label={right:$w_3$}] at (w3) {};
        \node[vertex,label={right:$w_4$}] at (w4) {};
        
        \node[encircle] at (w6) {};
        \node[encircle] at (v1) {};
        \node[encircle] at (v2) {};
        \node[encircle] at (v3) {};
        \node[encircle] at (v4) {};
        \node[encircle] at (v5) {};
        \node[encircle] at (w2) {};

        \node[ensquare] at (w6) {};
        \node[ensquare] at (v1) {};
        \node[ensquare] at (v2) {};
        \node[ensquare] at (v3) {};
        \node[ensquare] at (v4) {};
        \node[ensquare] at (v5) {};
        \node[ensquare] at (w4) {};
      \end{tikzpicture}
      \caption{A planar $(5,6)$-dwheel from the proof of
        \Claimref{weaksys_fs_dwheel}.  The encircled vertices form an
        extended $5$-wheel $\widehat{W}_1$ contained in the link $X_p$
        of the vertex $p$.  The ensquared vertices form an extended
        $5$-wheel $\widehat{W}_1'$ contained in $X_q$.}
      \figlabel{p_five_six_dwheel}
    \end{figure}

    By fullness of $W_2$, we have $w_2 \neq w_6$, $w_2 \nsim w_6$ and
    $w_2 \nsim v_3$.  Then, to show that $\widehat{W}_{1}$ is an
    extended $5$-wheel, it suffices to show $w_{2} \nsim v_4$ and
    $w_2 \nsim v_5$.  But $w_{2} \sim v_4$ would introduce the
    $4$-cycle $(w_2,v_2,w_6,v_4)$ and $w_2 \sim v_5$ would introduce
    the $4$-cycle $(w_2,v_2,w_6,v_5)$.  Applying $5$-largeness
    to either of these $4$-cycles would contradict fullness of either
    $W_1$ or $W_2$.  Thus $\widehat{W}_{1}$ is an extended $5$-wheel
    and so is contained in the link $X_p$ of a vertex $p$ of $X$.
    Fullness of $W_2$ ensures that $p$ is distinct from $w_3$ and
    $w_4$ so that $p \notin W = W_1 \cup W_2$.  By symmetry, the same
    argument can be applied to $\widehat{W}_1'$ to show that
    $\widehat{W}_{1}' \subset X_{q}$ for some vertex $q$ of
    $X \setminus W$.

    If $p=q$ then we could apply $5$-largeness to the $4$-cycle
    $(w_{2},w_{3},w_{4},p)$ and then the fullness of $W_2$ would imply
    that $p \sim w_3$.  Then we would have $W \subset X_p$, which
    would prove the claim.  So we may assume that $p \neq q$.  Then
    fullness of $W_1$ and $5$-largeness applied to the $4$-cycle
    $(v_1,p,v_3,q)$ implies that $p \sim q$.

    \begin{figure}[h]
      \centering
      \begin{tikzpicture}
        \coordinate (w6) at (-1,0);
        \coordinate (v2) at (1,0);
        \coordinate (v3) at (0,-4/3);
        \coordinate (v1) at (0,4/3);
        \coordinate (v4) at (-2,-3/4);
        \coordinate (v5) at (-2,3/4);
        \coordinate (w3) at (3,0);
        \coordinate (w2) at (2,4/3);
        \coordinate (w4) at (2,-4/3);
        \coordinate (p) at (0,1/2);
        \coordinate (q) at (0,-1/2);

        \begin{scope}[dotted]
          \draw (v1) -- (v2) -- (v3) -- (v4) -- (v5) -- (v1);
          \draw (w6) -- (v1);
          \draw (w6) -- (v2);
          \draw (w6) -- (v3);
          \draw (w6) -- (v4);
          \draw (w6) -- (v5);
          \draw (v1) -- (w2);
          \draw (w4) -- (v3);
        \end{scope}
        
        \begin{scope}[thick]
          \draw (w2) -- (w3) -- (w4) -- (q) -- (p) -- (w2);
          \draw (v2) -- (w2);
          \draw (v2) -- (w3);
          \draw (v2) -- (w4);
          \draw (v2) -- (p);
          \draw (v2) -- (q);
          \draw (v5) -- (p);
          \draw (v5) -- (q);
        \end{scope}
        
        \node[minivertex,label={left:$w_6$}] at (w6) {};
        \node[minivertex,label={above:$v_1=w_1$}] at (v1) {};
        \node[vertex,label={below:$v_2$}] at (v2) {};
        \node[minivertex,label={below:$v_3=w_5$}] at (v3) {};
        \node[minivertex,label={left:$v_4$}] at (v4) {};
        \node[vertex,label={left:$v_5$}] at (v5) {};
        \node[vertex,label={right:$w_2$}] at (w2) {};
        \node[vertex,label={right:$w_3$}] at (w3) {};
        \node[vertex,label={right:$w_4$}] at (w4) {};
        \node[vertex,label={above:$p$}] at (p) {};
        \node[vertex,label={below:$q$}] at (q) {};
      \end{tikzpicture}
      \caption{The extended $5$-wheel $\widehat{W}$ from the proof of
        \Claimref{weaksys_fs_dwheel}.}
      \figlabel{ws_five_six_ext_fwheel}
    \end{figure}
    
    Next we need to show that
    $\widehat{W} = (v_{2};p,w_{2},w_{3},w_{4},q;v_{5})$ is an extended
    $5$-wheel. See \Figref{ws_five_six_ext_fwheel}.  Fullness of $W_1$
    implies that $v_5 \neq v_2$ and $v_5 \nsim v_2$.  We need to show
    that $v_5$ is not adjacent to $w_2$, $w_3$ or $w_4$.  But
    adjacency of $v_5$ to any of these vertices would introduce a
    $4$-cycle in $W$ two which applying $5$-largeness
    would either contradict fullness of $W_1$ or fullness of $W_2$.
    Thus $\widehat{W}$ is an extended $5$-wheel and so is contained in
    the link $X_r$ of a vertex $r$ of $X$.  Fullness of $W_1$ and
    $W_2$ ensure that $r \notin W$.

    Since $v_5 \nsim v_2$, the $5$-largeness of $X$ applied to $4$-cycles
    $(v_5,r,v_2,v_1)$ and $(v_5,r,v_2,w_6)$ imply that $r$ is adjacent
    to $v_1$ and $w_6$.  Then, similarly, the $4$-cycle
    $(w_6,r,w_4,v_3)$ implies that $r \sim v_3$.  Finally, the
    $4$-cycle $(v_5,r,v_3,v_4)$ implies that $r \sim v_4$.  Thus
    $W \subset B_1(r)$.
\end{subproof}

The above three claims together imply that $X$ is $7$-located.
\end{proof}

\section{A $\CAT(0)$ metric for $7$-located disc diagrams}
\seclabel{cat0_discs}

In this section we show that the triangles of any $7$-located disc $D$
can be metrized in a natural way as Euclidean triangles such that $D$
is $\CAT(0)$.  Only three isometry types of triangles are required and
the choice of isometry type for each $2$-simplex $\sigma$ will depend
only on the local combinatorics around $\sigma$.

\begin{lem}\lemlabel{seven_loc_disc_degrees}
  Let $D$ be a $7$-located disc.  Then any pair of adjacent internal
  vertices $u,v$ of $D$ satisfy $\deg(v) + \deg(w) \ge 12$.
\end{lem}
\begin{proof}
  Since $D$ is a simplicial complex and $v$ is an interior vertex, it
  is the center of a full wheel $W_v$.  Similarly, we have a full
  wheel $W_w$ in $D$ with center $w$ and $W_v \cup W_w$ is a planar
  dwheel of boundary length $\deg(v) + \deg(w) - 4$.  Since $D$ is
  planar, the dwheel $W_v \cup W_w$ cannot be contained in a $1$-ball.
  Thus $\deg(v) + \deg(w) - 4 \ge 8$.
\end{proof}

\begin{defn}
  Let $W$ be a $k$-wheel with central vertex $v$.  We metrize each
  triangle of $W$ as a Euclidean triangle whose angle at $v$ is
  $\frac{2\pi}{k}$, whose remaining two angles are both equal to
  $\frac{(k-2)\pi}{2k}$ and whose boundary edge $e \subset \bd W$ has
  length $1$.  We then metrize $W$ as a Euclidean polygonal complex.
  We call the resulting metric the \emph{flattened wheel metric} on
  $W$.
\end{defn}

\begin{rmk}
  A $k$-wheel $W$ with the flattened wheel metric is isometric to a
  regular Euclidean $k$-gon of side length $1$.  Under this isometry,
  the central vertex is sent to the center of the $k$-gon.
\end{rmk}

\begin{defn}\defnlabel{flattened_wheel_metric}
  Let $D$ be a simplicial disc such that for any $k_1$-wheel
  $W_1 \subset D$ and any $k_2$-wheel $W_2 \subset D$, if $k_1 < 6$
  and $k_2 < 6$ then either $W_1 = W_2$ or
  $W_1 \cap W_2 \subset \bd W_1 \cap \bd W_2$.  The \emph{flattened
    wheel metric} on $D$ is the Euclidean polygonal complex metric
  obtained by metrizing each $k$-wheel of $D$ for which $k < 6$ with
  the flattened wheel metric and metrizing any remaining triangles as
  regular Euclidean triangles of side length $1$.  With this metric,
  the \emph{flattened wheels} of $D$ are its $k$-wheels with $k < 6$.
\end{defn}

\begin{lem}
  Let $D$ be a $7$-located disc.  Let $W_1$ and $W_2$ be distinct
  wheels of $D$ of boundary length less than $6$.  Then
  $W_1 \cap W_2 \subset \bd W_1 \cap \bd W_2$.  In particular, the
  flattened wheel metric is well-defined for $D$.
\end{lem}
\begin{proof}
  If $W_1$ and $W_2$ intersect in their interiors then they must share
  a triangle $\sigma$.  Thus the central vertices $v_1$ and $v_2$ of
  $W_1$ and $W_2$ are adjacent.  But then, by
  \Lemref{seven_loc_disc_degrees}, we have
  $|\bd W_1| + |\bd W_2| = \deg(v_1) + \deg(v_2) \ge 12$ so that one
  of the wheels must have boundary length at least $6$.
\end{proof}

In the next theorem we establish a connection between some $7$-located discs and $\CAT(0)$ spaces.

\begin{thm}\label{5.6}
  Let $D$ be a $7$-located simplicial disc endowed with the flattened
  wheel metric.  Then $D$ is $\CAT(0)$.
\end{thm}
\begin{proof}
  Since $D$ is simply connected, it suffices to show that the sum of
  the angles of the corners of triangles incident to any interior
  vertex $v$ of $D$ is at least $2\pi$ \cite[Theorem~II.5.4 and
  Lemma~II.5.6]{Bridson:1999}.  Because $D$ is flag, it has no
  interior vertices of degree $3$.  Any interior vertex of degree $4$
  or $5$ is the central vertex of a flattened wheel and so has angle
  sum exactly $2\pi$.  Away from flattened wheel centers, corner
  angles of triangles are either $\frac{\pi}{4}$, $\frac{3\pi}{10}$ or
  $\frac{\pi}{3}$.  Thus any interior vertex of degree at least $8$
  has angle sum at least $2\pi$.

  Thus we may assume that $v$ is an interior vertex of degree $6$ or
  $7$.  If $v$ is not incident to any flattened wheel then its
  incident triangles are all regular Euclidean and so have corner
  angle $\frac{\pi}{3}$.  Thus $v$ has angle sum
  $\frac{6\pi}{3} = 2\pi$ or $\frac{7\pi}{3} > 2\pi$.  So we may
  assume that $v$ is incident to a flattened wheel.  Since
  $\deg(v) \le 7$, by \Lemref{seven_loc_disc_degrees}, any central
  vertex of a flattened wheel $W$ incident to $v$ has degree $5$.
  Then $v$ is incident to a vertex of degree $5$ and any triangle
  corner incident to $v$ has angle at least $\frac{3\pi}{10}$.  So, by
  \Lemref{seven_loc_disc_degrees}, we have $\deg(v) = 7$ and so $v$
  has angle sum at least $\frac{21\pi}{10} > 2\pi$.
\end{proof}

\begin{cor}\corlabel{isoperfunc}
  There is a uniform quadratic function $f \colon \N \to \N$ such that
  any $7$-located disc $D$ of boundary length $|\bd D| = n$ has at
  most $f(n)$ triangles.
\end{cor}
\begin{proof}
  The area of a minimal disc diagram for a loop of length $n$ in a
  $\CAT(0)$ space is bounded by the area of a circle of circumference
  $n$ in the Euclidean plane \cite[Theorem~III.2.17]{Bridson:1999}.
  It follows that if $D$ is a $7$-located disc of boundary length $n$,
  it has area at most $\frac{n^2}{4\pi}$ when endowed with the
  flattened wheel metric.  Each triangle of $D$ has area at least
  $\frac{1}{4}$ so the number of triangles of $D$ is at most
  $f(n) = \frac{n^2}{\pi}$.
\end{proof}

\section{A Minimal Disc Diagram Lemma for $7$-located locally
  $5$-large complexes}
\seclabel{minimal_discs}

In this section we prove that a minimal area disc diagram in a
$7$-located locally $5$-large complex is $7$-located.

\begin{rmk}\label{5.4}
  The simplicial complex $X$ in \Figref{locally_5_large_necessary} is
  $7$-located but it is not locally $5$-large.  However, it has a
  minimal area disc diagram $(D,f)$ that is not $7$-located.  This
  example justifies the additional hypothesis of local $5$-largeness
  for the minimal disc diagram lemma.

\end{rmk}

\begin{figure}[h]
  \begin{center}
    \includegraphics[height=2cm]{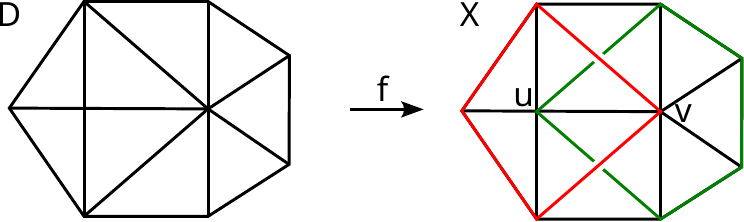}
    \caption{}
    \figlabel{locally_5_large_necessary}
  \end{center}
\end{figure}

We shall refer to the following lemmas frequently in the rest of this
section.

\begin{lem}\label{5.1}
  Let $X$ be a flag simplicial complex and let $\gamma$ be a
  homotopically trivial loop in $X$. Let $(D,f)$ be a disc diagram for
  $\gamma$.  Let $(u,a,v,b)$ be a $4$-cycle in $X$.  If $f(u) = f(v)$
  then there exists a disc diagram $(D',f')$ for $\gamma$ of lesser
  area than $D$.
\end{lem}

\begin{proof}
  The subdisc bounded by $(u,a,v,b)$ has at least two triangles.  We
  delete this subdisc. We glue $u$ to $v$, the edge
  $\langle a,u \rangle$ to the edge $\langle a,v \rangle$, and the
  edge $\langle b,u \rangle$ to the edge $\langle b,v \rangle$. Thus
  we get a new disc $D'$.  See \Figref{fourcycle_subdisc_replacement}.
  Let $f'$ be the map induced by the gluing. This is well defined
  since $f(u) = f(v)$,
  $f(\langle a,u \rangle) = f(\langle a,v \rangle)$ and
  $f(\langle b,u \rangle) = f(\langle b,v \rangle)$.  Thus $(D',f')$
  is a disc diagram for $\gamma$ of lesser area than $D$.

   \begin{figure}[h]
            \begin{center}
               \includegraphics[height=5cm]{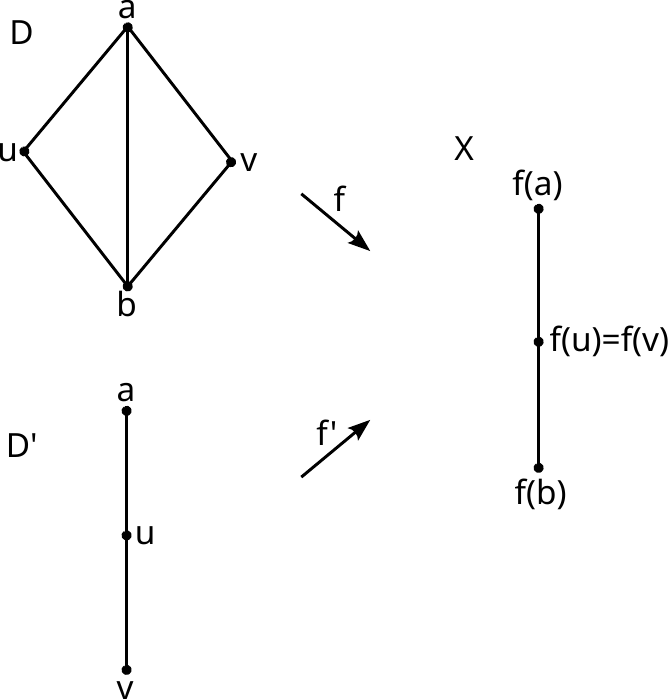}
               \caption{}
               \figlabel{fourcycle_subdisc_replacement}
            \end{center}
        \end{figure}
\end{proof}

\begin{lem}\label{5.2}
  Let $X$ be a flag simplicial complex and let $\gamma$ be a
  homotopically trivial loop in $X$. Let $(D,f)$ be a disc diagram for
  $\gamma$. Let $u,v$ be vertices of a $5$-cycle
  $\alpha = (a,v,b,c,u)$ of $D$ with $f(u)=f(v)$. Then there exists a
  disc diagram $(D',f')$ for $\gamma$ of lesser area than $D$.
\end{lem}

\begin{proof}
  The subdisc bounded by $\alpha$ has at least three $2$-simplices.
  We delete this subdisc. We obtain a new disc diagram $D'$ by glueing
  $u$ to $v$, gluing the edge $\langle a,u \rangle$ to the edge
  $\langle a,v \rangle$ and gluing a $2$-simplex
  $\langle u,b,c \rangle$.  See
  \Figref{fivecycle_subdisc_replacement}.  Let $f'$ be the map induced
  by the gluing. This is well defined since $f(u) = f(v)$,
  $f(\langle a,u \rangle) = f(\langle a,v \rangle)$ and, by flagness,
  $\langle f(u),f(b),f(c) \rangle$ is a triangle in $X$.  Thus
  $(D',f')$ is a disc diagram for $\gamma$ of lesser area than $D$.

 \begin{figure}[h]
            \begin{center}
               \includegraphics[height=4cm]{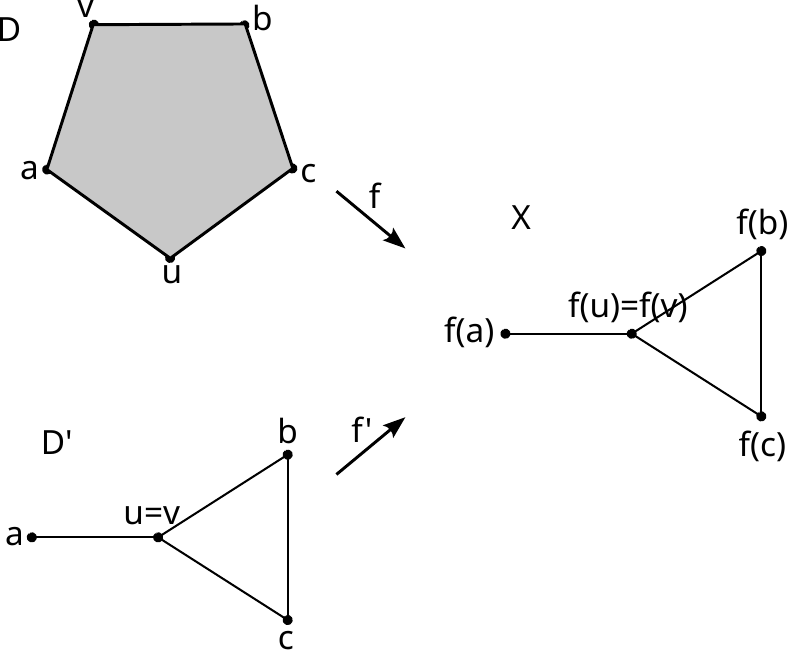}
               \caption{}
               \figlabel{fivecycle_subdisc_replacement}
            \end{center}
        \end{figure}

\end{proof}

\begin{lem}\label{5.3}
  Let $X$ be a flag simplicial complex and let $\gamma$ be a
  homotopically trivial loop in $X$. Let $(D,f)$ be a disc diagram for
  $\gamma$. Let $u,v$ be vertices of a $6$-cycle
  $\alpha=(u,a,b,v,c,d)$ of $D$ with $d(u,v)=3$ such that
  $f(u)=f(v)$. Then there exists a disc diagram $(D',f')$ for $\gamma$
  of lesser area than $D$.
\end{lem}

\begin{proof}
  The subdisc bounded by $\alpha$ has at least four $2$-simplices.  We
  obtain a new disc diagram $D'$ by deleting the subdisc bounded by
  $\alpha$, gluing $u$ to $v$ and then gluing in two triangles
  $\langle a,b,u \rangle$ and $\langle u,c,d \rangle$.  See
  \Figref{sixcycle_subdisc_replacement}.  Let $f'$ be the map induced
  by gluing. This is well defined since $f(u) = f(v)$ and, by
  flagness, $\langle f(a),f(b),f(u) \rangle$,
  $\langle f(u),f(c),f(d) \rangle$ are triangles in $X$. Hence
  $(D',f')$ is a disc diagram for $\gamma$ of lesser area than $D$.
  
 \begin{figure}[h]
            \begin{center}
               \includegraphics[height=4cm]{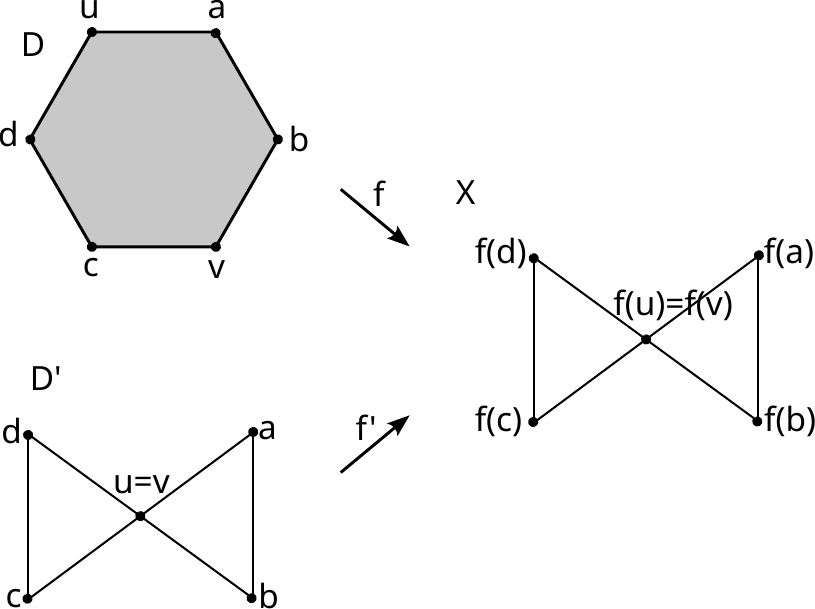}
               \caption{}
               \figlabel{sixcycle_subdisc_replacement}
            \end{center}
        \end{figure}
\end{proof}

\begin{lem}\label{no4wheel}
  Let $X$ be a flag and locally $5$-large simplicial complex.  Let
  $(D,f)$ be a minimal area disc diagram for a cycle $\gamma$ in $X$.
  Then $D$ does not contain a $4$-wheel.
\end{lem}
\begin{proof}
  Suppose $W$ is a $4$-wheel in $D$.  If the restriction $f|_W$ is not
  injective then some pair of antipodal vertices $u,v \in \partial W$
  have a common image $f(u) = f(v)$.  This contradicts minimality by
  Lemma~\ref{5.1}.

  If the restriction $f|_W$ is injective then, since $X$ is locally
  $5$-large, a pair of antipodal vertices $u,v \in \partial W$ have
  images $f(u)$ and $f(v)$ that are joined by an edge $e$ in $X$.
  Since $X$ is flag, the image $f(\partial W)$ along with $e$ span two
  triangles so that $f(\partial W)$ has a disc diagram of area $2$.
  Thus we may remove the interior of $W$ in $D$ and replace it with an
  edge joining $u$ and $v$ and a pair of triangles to obtain a disc
  diagram for $\gamma$ of lesser area, contradicting minimality of
  $D$.
\end{proof}

\begin{lem}\label{5wheelsfull}
  Let $X$ be a flag and locally $5$-large simplicial complex.  Let
  $(D,f)$ be a minimal area disc diagram for a cycle $\gamma$ in $X$
  and let $W$ be a $5$-wheel in $D$.  Then the image $f(W)$ is a full
  $5$-wheel of $X$.
\end{lem}
\begin{proof}
  It follows from Lemma~\ref{5.2} that $f|_W$ is injective.  Then
  $f(W)$ is a $5$-wheel in $X$.  We need only verify that it is full.
  If not then some pair of vertices $u,v \in \partial W$ that are not
  adjacent in $\partial W$ have images $f(u)$ and $f(v)$ that are
  joined by an edge $e$.  Since $X$ is flag, the edge $e$ spans a
  triangle with each of $f(u)$ and $f(v)$.  Let $a$ be the unique
  vertex of $\partial W$ adjacent to both $u$ and $v$.  Let $c$ be the
  central vertex of $W$.  Then $(u,a,v,c)$ is a $4$-cycle bounding a
  subdisc of $D$ with two triangles.  We delete this subdisc, join $u$
  and $v$ by an edge $e'$ and add triangles $\langle a,u,v\rangle$ and
  $\langle c,u,v \rangle$ to obtain a new disc diagram $D'$ for
  $\gamma$ of the same (minimal) area.  In particular, $D'$ is a
  minimal area disc diagram.  But the triangle $\langle c,u,v \rangle$
  along with the three remaining original triangles of $W$ form a
  $4$-wheel, which contradicts minimality by Lemma~\ref{no4wheel}.
\end{proof}

\begin{lem}\label{7loc5larg_injective_dwheel}
  Let $X$ be a $7$-located locally $5$-large simplicial complex.  Let
  $(D,f)$ be a minimal area disc diagram for a cycle $\gamma$ in X.
  Let $W$ be a $(5,5)$- or $(5,6)$-dwheel of $D$.  Then the
  restriction $f|_W$ is not injective.
\end{lem}
\begin{proof}
  For the sake of deriving a contradiction, we assume that $f|_W$ is
  injective.  Since $D$ is planar, the dwheel $W$ is also planar.  We
  consider first the case where $W$ is a planar $(5,5)$-dwheel.  By
  Lemma~\ref{5wheelsfull}, the $(5,5)$-dwheel $f(W)$ has full wheels.
  Then, since $X$ is $7$-located, the image $f(W)$ is contained in the
  $1$-ball centered at a vertex $v$ of $X$.  Since $X$ is flag, it
  follows that $f(\partial W)$ has a disc diagram with at most $6$
  triangles, contradicting minimality of $D$.

  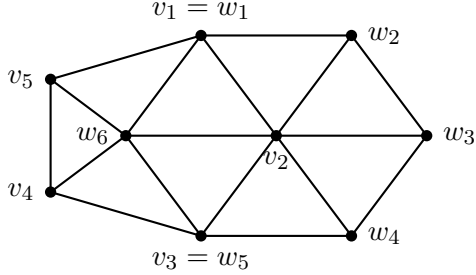
\begin{figure}[h]
    \centering
    \begin{tikzpicture}
      \coordinate (w6) at (-1,0);
      \coordinate (v2) at (1,0);
      \coordinate (v3) at (0,-4/3);
      \coordinate (v1) at (0,4/3);
      \coordinate (v4) at (-2,-3/4);
      \coordinate (v5) at (-2,3/4);
      \coordinate (w3) at (3,0);
      \coordinate (w2) at (2,4/3);
      \coordinate (w4) at (2,-4/3);

      \begin{scope}[thick]
        \draw (v1) -- (v2) -- (v3) -- (v4) -- (v5) -- (v1);
        \draw (w6) -- (v1);
        \draw (w6) -- (v2);
        \draw (w6) -- (v3);
        \draw (w6) -- (v4);
        \draw (w6) -- (v5);
        \draw (v1) -- (w2) -- (w3) -- (w4) -- (v3);
        \draw (v2) -- (w2);
        \draw (v2) -- (w3);
        \draw (v2) -- (w4);
      \end{scope}
      
      \node[vertex,label={left:$w_6$}] at (w6) {};
      \node[vertex,label={above:$v_1=w_1$}] at (v1) {};
      \node[vertex,label={below:$v_2$}] at (v2) {};
      \node[vertex,label={below:$v_3=w_5$}] at (v3) {};
      \node[vertex,label={left:$v_4$}] at (v4) {};
      \node[vertex,label={left:$v_5$}] at (v5) {};
      \node[vertex,label={right:$w_2$}] at (w2) {};
      \node[vertex,label={right:$w_3$}] at (w3) {};
      \node[vertex,label={right:$w_4$}] at (w4) {};
    \end{tikzpicture}
    \caption{A planar $(5,6)$-dwheel.}
    \figlabel{five_six_dwheel}
  \end{figure}

  We now consider the case where $W$ is a planar $(5,6)$-dwheel.  Let
  $W = W_{1} \cup W_{2} = (w_{6};v_{1},v_{2},v_{3},v_{4},v_{5}) \cup
  (v_{2};w_{1}=v_{1},w_{2},w_{3},$ $w_{4},w_{5}=v_{3},w_{6})$; see
  \Figref{five_six_dwheel}.  By Lemma~\ref{5wheelsfull}, the image
  $f(W_1)$ is a full $5$-wheel.  If the $6$-wheel $f(W_2)$ is also a
  full then we can argue as in the $(5,5)$-dwheel case that
  $f(\partial W)$ has a disc diagram with at most $7$ triangles, which
  leads to a contradiction with the minimality of $D$.  Thus we need
  only consider the case where $f(W_2)$ is not full.  Then some pair
  of vertices $u,v \in \partial W_2$ that are not adjacent in
  $\partial W_2$ have images $f(u)$ and $f(v)$ that are adjacent.

  \begin{claim}
    Either $\{u,v\} = \{w_6,w_2\}$ or $\{u,v\} = \{w_6,w_4\}$.
  \end{claim}
  \begin{subproof}
    We first rule out the possibility that $u$ and $v$ are antipodal
    in $\partial W_2$.  If this were so, we could cut $D$ open along
    the path $(u,v_2,w)$ and fill the resulting boundary path with an
    edge joining $u$ and $w$ (mapping to the edge
    $\langle f(u), f(v) \rangle$) and a pair of triangles (which map
    to triangles in $X$ by flagness).  In the resulting disc diagram
    for $\gamma$, the cycle $\partial W_2$ bounds a disc with two
    $4$-wheels.  But, by arguments similar to those in the proof of
    Lemma~\ref{no4wheel}, each $4$-wheel boundary (as a cycle in $X$)
    has a disc diagram with at most two triangles.  Thus
    $\partial W_2$ has a disc diagram with at most four triangles,
    contradicting minimality of $(D,f)$.

    Thus $u$ and $v$ are at distance $2$ in $\partial W_2$.  Since
    $f(W_1)$ is full in $X$, we have $\{u,v\} \neq \{w_1,w_5\}$.  It
    remains to rule out the remaining possibilities:
    $\{u,v\} = \{w_1,w_3\}$, $\{u,v\} = \{w_2,w_4\}$ and
    $\{u,v\} = \{w_3,w_5\}$.  In these cases, we perform a disc
    diagram surgery which turns $v_2$ into the center of a $5$-wheel,
    similar to the surgery performed in the proof of
    Lemma~\ref{5wheelsfull} that produced a $4$-wheel.  The resulting
    disc diagram has the same (minimal) area but has a $5$-wheel and
    thus the argument reduces to the $(5,5)$-dwheel case.
  \end{subproof}

  Thus the additional edges in the full subcomplex induced on the
  $6$-wheel $f(W_2)$ are either $\langle f(w_6),f(w_2) \rangle$ or
  $\langle f(w_6),f(w_4) \rangle$ or both.  If just one of these is
  present then $f(W)$ spans a non-planar $(5,5)$-dwheel with full
  wheels so that $\partial W$ has a disc diagram with at most $7$
  triangles, contradicting minimality of $(D,f)$.  If, on the other
  hand, both edges are present then we perform two disc diagram
  surgeries as in the proof of Lemma~\ref{no4wheel} transforming $W$
  from a $(5,6)$-dwheel to a $(7,4)$-dwheel.  The resulting disc
  diagram has the same (minimal) area and yet contains a $4$-wheel,
  contradicting Lemma~\ref{no4wheel}.
\end{proof}

We prove next the minimal disc diagram lemma for $7$-located locally
$5$-large simplicial complexes.

\begin{thm}\label{5.7}
  Let $X$ be a $7$-located locally $5$-large simplicial complex. Let
  $\gamma$ be a homotopically trivial loop in $X$.  Any minimal area
  disc diagram $(D,f)$ for $\gamma$ is $7$-located and locally
  $5$-large.
\end{thm}
\begin{proof}
  By Lemma~\ref{no4wheel}, there are no $4$-wheels in $D$.  It follows
  that $D$ is locally $5$-large so we need only prove that $D$ is
  $7$-located.  Suppose for the sake of finding a contradiction that
  $D$ is not $7$-located.  Then $D$ has a dwheel $W = W_1 \cup W_2$ of
  boundary length at most 7 that is not contained in a $1$-ball.
  Since $D$ is planar, so is $W$.  Since $D$ has no $4$-wheels, the
  dwheel $W$ must be a planar $(5,5)$- or $(5,6)$-dwheel.  By
  Lemma~\ref{7loc5larg_injective_dwheel}, the restriction $f|_W$ is
  not injective.  Then there exist distinct vertices $v,w \in W$ such
  that $f(v)=f(w)$.

  We consider first the case where $v$ and $w$ belong to a common
  wheel $W_i$ of $W$.  Since $f$ is a nondegenerate map and $X$ is
  simplicial, the vertices $v$ and $w$ cannot be adjacent.  Thus they
  must be contained in the boundary $\partial W_i$ and be at distance
  $2$ or $3$ in $\partial W_i$.  Note that $W_i$ is necessarily a
  $6$-wheel in the case that $d_{\partial W_i}(v,w) = 3$.  But then we
  contradict minimality by Lemma~\ref{5.1} if
  $d_{\partial W_i}(v,w) = 2$ or by Lemma~\ref{5.3} if
  $d_{\partial W_i}(v,w) = 3$.

  We consider now the case where $v$ and $w$ do not belong to a common
  wheel of $W$.  It will be helpful here to name the vertices of $W$
  so let
  $W = W_1 \cup W_2 = (w_k;v_1,v_2, \dots, v_5) \cup (v_2;
  w_1=v_1,w_2, \dots, w_{k-1} = v_3, w_k)$ with $k \in \{5,6\}$.  We
  first consider the $(5,5)$-dwheel case, i.e., the case where $k=5$.
  See \Figref{p_five_five_dwheel}.  Up to symmetry, there are only two
  cases to consider: $(v,w) = (v_4,w_3)$ and $(v,w) = (v_4,w_2)$.  In
  the case $(v,w) = (v_4,w_3)$, the $5$-cycle $(v_4,v_3,w_3,v_2,w_5)$
  contradicts minimality by Lemma~\ref{5.2}.  In the case
  $(v,w) = (v_4,w_2)$, the $6$-cycle
  $\partial W = (v_4,v_3,w_3,w_2,w_1,v_5)$ contradicts minimality by
  Lemma~\ref{5.3}.

  We now consider the $(5,6)$-dwheel case, i.e., the case where $k=6$.
  See \Figref{five_six_dwheel}.  Up to symmetry, there are three cases
  to consider: $(v,w) = (v_4,w_4)$, $(v,w) = (v_4,w_3)$ and
  $(v,w) = (v_4,w_2)$.  As with the $(5,5)$-dwheel, in each case we
  contradict minimality by applying either Lemma~\ref{5.2} to a
  $5$-cycle or Lemma~\ref{5.3} to a $6$-cycle.  Specifically, for the
  case $(v,w) = (v_4,w_4)$ we consider the $5$-cycle
  $(v_4,v_3,w_4,v_2,w_6)$, for the case $(v,w) = (v_4,w_3)$ we
  consider the $6$-cycle $(v_4,v_3,w_4,w_3,v_2,w_6)$ and for the
  $(v,w) = (v_4,w_2)$ case we consider the $6$-cycle
  $(v_4,w_6,v_2,w_2,w_1,v_5)$.

\end{proof}

\bibliographystyle{abbrv}
\bibliography{refs}

\end{document}